\documentclass{amsart}
\usepackage{amssymb}
\usepackage{amsmath}
\usepackage[all,knot,cmtip]{xy}
\usepackage{mathrsfs}

\numberwithin{equation}{section}

\theoremstyle{plain}   
\newtheorem{bigthm}{Theorem}   

\newtheorem{theorem}[equation]{Theorem}  
\newtheorem{cor}[equation]{Corollary}     
\newtheorem{lemma}[equation]{Lemma}         
\newtheorem{prop}[equation]{Proposition}

\theoremstyle{definition}

\theoremstyle{remark}
\newtheorem{remark}[equation]{Remark}

\newcommand{\Tor}{\operatorname{Tor}}

\newcommand{\Hom}{\operatorname{Hom}}

\newcommand{\TF}{\operatorname{TF}}
\newcommand{\TC}{\operatorname{TC}}
\newcommand{\TR}{\operatorname{TR}}
\newcommand{\THH}{\operatorname{THH}}

\newcommand{\HC}{\operatorname{HC}}

\newcommand{\Z}{\mathbb{Z}}
\newcommand{\Q}{\mathbb{Q}}

\newcommand{\Zp}{\mathbb{Z}_p}

\newcommand{\Fp}{\mathbb{F}_p}

\newcommand{\holim}{\operatornamewithlimits{holim}}

\newcommand{\id}{\operatorname{id}}
\newcommand{\pr}{\operatorname{pr}}
\newcommand{\tr}{\operatorname{tr}}

\newcommand{\xto}{\xrightarrow}

\begin{document}

\title{On relative and bi-relative algebraic $K$-theory \\
of rings of finite characteristic} 

\author{Thomas Geisser}

\address{University of Southern California, California, USA}

\email{geisser@usc.edu}

\author{Lars Hesselholt}

\address{Nagoya University, Nagoya, Japan}

\email{larsh@math.nagoya-u.ac.jp}

\thanks{The authors were supported in part by the National Science
 Foundation}

\maketitle

\section*{Introduction}

Throughout, we fix a prime number $p$ and consider unital
associative rings in which $p$ is nilpotent. It was proved by
Weibel~\cite[Cor.~5.3, Cor.~5.4]{weibel} long ago that, for such
rings, the relative  $K$-groups associated with a nilpotent extension
and the bi-relative $K$-groups associated with a Milnor square are
$p$-primary torsion groups. However, the question of whether these
groups can contain a $p$-divisible torsion subgroup has remained an
open and intractable problem. In this paper, we answer this question
in the negative. In effect, we prove the stronger statement that the
groups in question are always $p$-primary torsion groups of bounded
exponent.

In more detail, let $A$ be a unital associative ring and let $I
\subset A$ be a two-sided ideal. Then one defines the relative
$K$-theory spectrum $K(A,I)$ to be the mapping fiber of the map of
$K$-theory spectra $K(A) \to K(A/I)$ induced by the canonical
projection. Hence, there is a long-exact sequence of homotopy groups
$$\cdots \to K_q(A,I) \to K_q(A) \to K_q(A/I) \xto{\partial}
K_{q-1}(A,I) \to \cdots.$$
Here and throughout, $K(A)$ denotes the non-connective Bass completed
algebraic $K$-theory spectrum of the ring
$A$~\cite[Def.~6.4]{thomasontrobaugh}. We prove the following result.

\begin{bigthm}\label{relative}Let $A$ be a unital associative ring and
$I \subset A$ a two-sided nilpotent ideal. Suppose that the prime
number $p$ is nilpotent in $A$. Then, for every integer $q$, the
relative group $K_q(A,I)$ is a $p$-primary torsion group of bounded
exponent.
\end{bigthm}

Thm.~\ref{relative} implies, in particular, that the $p$-completion
map is an isomorphism 
$$K_q(A,I) \xto{\sim} K_q(A,I;\Zp).$$
In general, however, the exponent of $K_q(A,I)$ depends on the degree
$q$. For instance, if $A = \Fp[x]/(x^m)$ and $I=(x)$, the exponent
tends to infinity with $q$~\cite{hm1}. From Thm.~\ref{relative}
together with a theorem of McCarthy~\cite[Main Thm.]{mccarthy1} and a
previous theorem of our own~\cite[Thm.~2.1.1]{gh3}, we conclude the
following result:

\begin{bigthm}\label{relativetheorem}Let $A$ be a unital associative
ring, and let $I \subset A$ be a nilpotent two-sided ideal. Suppose
that the prime number $p$ is nilpotent in $A$. Then, for every integer
$q$, the cyclotomic trace map induces an isomorphism of pro-abelian
groups
$$K_q(A,I) \xto{\sim} \{ \TC_q^n(A,I;p) \}.$$
The pro-abelian group on the right-hand side is indexed by positive
integers $n$.
\end{bigthm}

We remark that Thm.~\ref{relativetheorem} is equivalent to the
statement that, for $n$ large, the group $K_q(A,I)$ embeds as a direct
summand in $\TC_q^n(A,I;p)$ and that the limit system of cokernels
satisfies the Mittag-Leffler condition and has limit zero. 

Let $f \colon A \to B$ is a map of unital associative rings and $I$ a
two-sided ideal of $A$ that is mapped isomorphically onto a 
two-sided ideal of $B$. Then one defines the bi-relative
$K$-theory spectrum $K(A,B,I)$ to be the mapping fiber of the map of
relative $K$-theory spectra $K(A,I) \to K(B,f(I))$ induced by the map
$f$. It follows that there is a long-exact sequence of homotopy groups
$$\cdots \to K_q(A,B,I) \to K_q(A,I) \to K_q(B,f(I)) \xto{\partial}
K_{q-1}(A,B,I) \to \cdots.$$
We prove the following result. 

\begin{bigthm}\label{birelative}Let $f \colon A \to B$ be a map of
unital associative rings, let $I \subset A$ be a two-sided ideal and
assume that $f \colon I \to f(I)$ is an isomorphism onto a two-sided
ideal of $B$. Suppose that the prime number $p$ is nilpotent in
$A$. Then, for every integer $q$, the bi-relative group $K_q(A,B,I)$
is a $p$-primary torsion group of bounded exponent.
\end{bigthm}

We remark again that Thm.~\ref{birelative} implies that completion map
$$K_q(A,B,I) \to K_q(A,B,I;\Zp)$$
is an isomorphism. The exponent of $K_q(A,B,I)$ in general depends on
the degree $q$. For example, if $A = \Fp[x,y]/(xy)$, $B = \Fp[x]
\times \Fp[y]$, and  $I = (x,y)$, the exponent tends to infinity with
$q$~\cite{h4}. Thm.~\ref{birelative} and our previous
theorem~\cite[Thm.~1]{gh4} implies the following result: 

\begin{bigthm}\label{birelativetheorem}Let $f \colon A \to B$ be a
map of unital associative rings, let $I \subset A$ be a
two-sided ideal and assume that $f \colon I \to f(I)$ is an
isomorphism onto a two-sided ideal of $B$. Suppose that the prime
number $p$ is nilpotent in $A$. Then, for every integer $q$, the
cyclotomic trace map
$$K_q(A,B,I) \to \{ \TC_q^n(A,B,I;p) \}$$
is an isomorphism of pro-abelian groups. The pro-abelian group on the
right-hand side is indexed by positive integers $n$.
\end{bigthm}

We again note that Thm.~\ref{birelativetheorem} is equivalent to the
statement that, for $n$ large, the group $K_q(A,B,I)$ is a direct
summand in $\TC_q^n(A,B,I;p)$ and that the limit system of cokernels
satisfies the Mittag-Leffler condition and has limit zero. 

Finally, we mention that the main motivation for the work reported in
this paper was the application of Thms.~\ref{relativetheorem}
and~\ref{birelativetheorem} to the proof in~\cite[Thm.~B]{gh6} that
the mapping fiber of the cyclotomic trace map
$$K(X) \to \{ \TC^n(X;p) \}$$
satisfies descent for the $\operatorname{cdh}$-topology on the 
category of schemes separated and essentially of finite type over an
infinite perfect field $k$ of positive characteristic $p$, provided
that resolution of singularities holds over $k$. The main advantage of
the functor $\{ \TC_q^n(-;p) \}$ that appears in the statement of
Thms.~\ref{relativetheorem} and~\ref{birelativetheorem} in comparison
to the functor $\TC_q(-;p)$ that appears in McCarthy's
theorem~\cite{mccarthy1} is that the former preserves filtered
colimits while the latter, in general, does not. Therefore, replacing
the latter functor by the former, the methods of sheaf cohomology
become available.

We recall that a pro-object of a category $\mathscr{C}$ is a functor
from a directed partially ordered set to the category $\mathscr{C}$
and that a strict map between two pro-objects with the same indexing
set is a natural transformation. A general map from a pro-object
$X = \{X_s\}_{s \in S}$ to a pro-object $Y = \{Y_t\}_{t \in T}$ is an
element of the set
$$\Hom_{\operatorname{pro}-\mathscr{C}}(X,Y)=
\operatornamewithlimits{lim}_T
\operatornamewithlimits{colim}_S\Hom_{\mathscr{C}}(X_s,Y_t).$$
In particular, a pro-object $X = \{X_s\}_{s \in S}$ in a category with
a null-object is zero if for all $s \in S$, there exists $s' \geq s$
such that the map $X_{s'} \to X_s$ is zero. We will often omit the
indexing set $S$ from the notation. It is then understood that the
indexing set $S$ is the range of all indices that are not assumed to
be fixed.

\section{Non-connective $K$-theory and the cyclotomic trace map}\label{nonconnectivesection}

In this section, we show that the cyclotomic trace map extends to a
map from the Bass completed non-connective $K$-theory to topological
cyclic homology. We first briefly review topological cyclic homology
and the cyclotomic trace map and refer to~\cite[Sect.~1]{hm3}
and~\cite{h3} for details.

Let $A$ be a unital associative ring. The topological Hochschild
spectrum $T(A)$ is a symmetric orthogonal $\mathbb{T}$-spectrum, where
$\mathbb{T}$ is the multiplicative group of complex numbers of modulus
$1$. Let $p$ be a prime number, and let $C_{p^{n-1}} \subset
\mathbb{T}$ be the subgroup of the indicated order. We define
$$\TR^n(A;p) = F((\mathbb{T}/C_{p^{n-1}})_+,T(A))^{\mathbb{T}}$$
to be the fixed point spectrum of the function $\mathbb{T}$-spectrum
$F((\mathbb{T}/C_{p^{n-1}})_+,T(A))$. Its homotopy groups are the
equivariant homotopy groups 
$$\TR_q^n(A;p) =
[ S^q \wedge (\mathbb{T}/C_{p^{n-1}})_+, T(A) ]_{\mathbb{T}}.$$
There are two maps of symmetric orthogonal spectra
$$R, F \colon \TR^n(A;p) \to \TR^{n-1}(A;p)$$
called the restriction and Frobenius maps. The symmetric orthogonal
spectrum $\TC^n(A;p)$ is defined to be the homotopy equalizer of the
maps $R$ and $F$ and the topological cyclic homology spectrum to be
the homotopy limit
$$\TC(A;p) = \holim \TC^n(A;p)$$
where the structure maps are the maps induced by the restriction
maps. We also consider the homotopy limits
$$\begin{aligned}
\TR(A;p) & = \holim_R \TR^n(A;p) \cr
\TF(A;p) & = \holim_F \TR^n(A;p) \cr
\end{aligned}$$
of the spectra $\TR^n(A;p)$ with respect to the restriction and
Frobenius maps, respectively. The Frobenius map induces a self-map of
$\TR(A;p)$, and $\TC(A;p)$ is canonically isomorphic to the homotopy
equalizer of this map and the identity map. Similarly, the restriction
map induces a self-map of $\TF(A;p)$ and $\TC(A;p)$ is canonically
isomorphic to the homotopy equalizer of this map and the identity
map. In particular, we have long-exact sequences of homotopy groups
$$\begin{aligned}
{} & \cdots \to \TC_q(A;p) \to \TR_q(A;p) \xto{\id - F} \TR_q(A;p)
\xto{\partial} \TC_{q-1}(A;p) \to \cdots \cr
{} & \cdots \to \TC_q(A;p) \to \TF_q(A;p) \xto{R - \id} \TF_q(A;p)
\xto{\partial} \TC_{q-1}(A;p) \to \cdots \cr
\end{aligned}
$$
It was proved in~\cite[Thm.~F]{hm1} that, if the ring $A$ is an
algebra over the commutative ring $k$, then the equivariant homotopy
groups $\TR_q^n(A;p)$ are modules over the ring $W_n(k)$ of Witt
vectors of length $n$ in $k$.

\begin{lemma}\label{TCTR}Let $A$ be a unital associative ring and
suppose that the prime number $p$ is nilpotent in $A$. Then, for all
integers $q$ and $n \geqslant 1$, the groups $\TR_q^n(A;p)$ and
$\TC_q^n(A;p)$ are $p$-primary torsion groups of bounded exponent.
\end{lemma}

\begin{proof}Suppose that $A$ is an $\Z/p^N\Z$-algebra.
Then~\cite[Thm.~F, Prop.~2.7.1]{hm1} show that the groups
$\TR_q^n(A;p)$ are $W_n(\Z/p^N\Z)$-modules, and therefore, are
annilated by multiplication by $p^{Nn}$. Finally, the long-exact
sequence
$$\cdots \longrightarrow
\TR_{q+1}^{n-1}(A;p) \xto{\,\;\partial\;\,}
\TC_q^n(A;p) \longrightarrow
\TR_q^n(A;p) \xto{R-F} 
\TR_q^{n-1}(A;p) \longrightarrow \cdots$$
shows that $\TC_q^n(A;p)$ is annihilated by $p^{N(2n-1)}$.
\end{proof}

In general, the groups $\TR_q(A;p)$ and $\TC_q(A;p)$ are not
$p$-primary torsion groups of bounded exponent. For example,
$\TC_0(\Fp;p) = \TR_0(\Fp;p) = \Zp$. 

We consider the diagram of canonical inclusions
$$\xymatrix{
{ A } \ar[r]^{f_{+}} \ar[d]^{f_{-}} &
{ A[t] } \ar[d]^{i_+} \cr
{ A[t^{-1}] } \ar[r]^{i_{-}} &
{ A[t^{\pm1}]. } \cr
}$$
The following result shows that $\TR_*^n(-;p)$ is a Bass complete
theory. 

\begin{prop}\label{TRbasscomplete}Let $A$ be a unital associative
ring. Then, for all prime numbers $p$, all integers $q$, and all
positive integers $n$, the sequence
$$0 \to \TR_q^n(A;p) \to 
\overset{ \displaystyle{ \TR_q^n(A[t];p) }}{ \underset{
\displaystyle{ \TR_q^n(A[t^{-1}];p) }}{ \oplus }} \to
\TR_q^n(A[t^{\pm1}];p) \xto{\partial_t} \TR_{q-1}^n(A;p) \to 0$$
where the left-hand map is $(f_{+}^*,-f_{-}^*)$, and where the middle map
is $i_{+}^*+i_{-}^*$, is exact. Moreover, the right-hand map
$\partial_t$ has a section given by multiplication by the image
$d\log[t]_n$ by the cyclotomic trace map of $t \in K_1(\Z[t^{\pm1}])$.
\end{prop}

\begin{proof}We first prove that the statement holds after localizing
the sequence in the statement at $p$. Let $r_{+} \colon A[t] \to A$
and $r_{-} \colon A[t^{-1}] \to A$ be the ring homomorphisms that map
$t$ and $t^{-1}$ to $0$. Since $r_{+} \circ f_{+}$ and $r_{-} \circ
f_{-}$ both are the identity map of $A$, these four maps give rise to
direct sum decompositions
$$\begin{aligned}
\TR_q^n(A[t];p) & = \TR_q^n(A;p) \oplus N^{+}\TR_q^n(A;p) \cr
\TR_q^n(A[t^{-1}];p) & = \TR_q^n(A;p) \oplus N^{-}\TR_q^n(A;p). \cr
\end{aligned}$$
The structure of the relative terms $N^{+}\TR_q^n(A;p)$ and 
$N^{-}\TR_q^n(A;p)$ was determined in~\cite[Thm.~B]{hm3}. The proof
given in loc.~cit.~also leads to a formula for the groups
$\TR_q^n(A[t^{\pm1}];p)$; we refer to~\cite[Thm.~2]{h7} for the
precise statement. By comparing the two formulas, we find that the
map
$$\TR_q^n(A;p) \oplus \TR_{q-1}^n(A;p) \oplus
N^{+}\TR_q^n(A;p) \oplus 
N^{-}\TR_q^n(A;p) \to \TR_q^n(A[t^{\pm1}];p)$$
that, on the first summand, is the map $\iota$ induced by the
$A$-algebra homomorphism $i_{+} \circ f_{+} = i_{-} \circ f_{-}$, on
the second summand is $\iota$ followed by multiplication by the image
by the cyclotomic trace of $t \in K_1(\Z[t^{\pm1}])$, and, on the
third and fourth summands, is the compositions of the canonical
inclusions of $N^{+}\TR_q^n(A;p)$ and $N^{-}\TR_q^n(A;p)$ and
$\TR_q^n(A[t^{-1}];p)$ in $\TR_q^n(A[t];p)$ and the maps induced by
the ring homomorphisms $i_{+}$ and $i_{-}$, respectively,
becomes an isomorphim after localizing at $p$. Hence, the
sequence of the statement becomes exact after localizing at $p$.

It remains to prove the statement after localizing the sequence in the
statement of the proposition away from $p$. We recall
from~\cite[Prop.~4.2.5]{hm1} that, for every unital associative ring
$B$, the map 
$$(R^{n-1-s}F^s) \colon \TR_q^n(B;p) \to \prod_{0 \leqslant s < n}
\TR_q^1(B;p)$$
becomes an isomorphism after localization away from $p$. Therefore, it
will suffice to prove the statement for $n = 1$. In this case, the
description of $\TR_q^n(A[t];p)$, $\TR_q^n(A[t^{-1}];p)$, and
$\TR_q^n(A[t^{\pm1}];p)$ that we recalled above is valid without
localizing at $p$; compare~\cite[Lemma~3.3.1]{hm3}. This completes
the proof. 
\end{proof}

The sequence of Quillen $K$-groups
$$0 \to K_q(A) \to 
\overset{ \displaystyle{ K_q(A[t]) }}{ \underset{
\displaystyle{ K_q(A[t^{-1}]) }}{ \oplus }} \to
K_q(A[t^{\pm1}]) \xto{\partial_t} K_{q-1}(A) \to 0$$
is exact, for $q$ positive. The Bass negative $K$-groups are
recursively defined so as to make the sequence exact, for all integers
$q$~\cite[Chap.~XII,~Sect.~7]{bass}. We recall from
Thomason-Trobaugh~\cite[Def.~6.4]{thomasontrobaugh} that the Bass
completion can be accomplished on the level of spectra. One obtains a
natural transformation
$$c \colon K^Q(A) \to K^B(A)$$
from the Waldhausen $K$-theory spectrum $K^Q(A)$ to a new spectrum
$K^B(A)$ whose homotopy groups in negative degrees are canonically
isomorphic to the Bass negative $K$-groups and the map of homotopy
groups induced by the map $c$ is an isomorphism in non-negative
degrees. We may similarly apply the Bass completion of
Thomason-Trobaugh to the topological cyclic homology functor. 

\begin{cor}\label{TCbasscomplete}The Bass completion map
$$c \colon \TC_q(A;p) \to \TC_q^B(A;p)$$
is an isomorphism, for all integers $q$.
\end{cor}

\begin{proof}It follows from Prop.~\ref{TRbasscomplete} that the
sequence 
$$0 \to \TR_q(A;p) \to 
\overset{ \displaystyle{ \TR_q(A[t];p) }}{ \underset{
\displaystyle{ \TR_q(A[t^{-1}];p) }}{ \oplus }} \to
\TR_q(A[t^{\pm1}];p) \xto{\partial_t} \TR_{q-1}(A;p) \to 0$$
is exact, for all integers $q$, and that the right-hand map
$\partial_t$ has a section given by multiplication by the image
$d\log[t]$ by the cyclotomic trace of $t \in K_1(\Z[t^{\pm1}])$. This
implies that the Bass completion map
$$c \colon \TR_q(A;p) \to \TR_q^B(A;p)$$
is an isomorphism, for all integers $q$. This, in turn, implies that
the Bass completion map of the statement is an isomorphism, for all
integers $q$. 
\end{proof}

The cyclotomic trace map is a natural map of symmetric spectra
$$\tr \colon K^Q(A) \to \TC(A;p).$$
It was originally defined by
B\"{o}kstedt-Hsiang-Madsen~\cite{bokstedthsiangmadsen}. However, a
technically better construction of this map was given by
Dundas-McCarthy~\cite[Sect.~2.0]{dundasmccarthy}. The latter
construction was used in~\cite[Appendix]{gh} to show that the
cyclotomic trace map is multiplicative. Since the Bass completion
of~\cite[Def.~6.4]{thomasontrobaugh} is functorial, we obtain a
commutative diagram of natural transformations 
$$\xymatrix{
{ K^Q(A) } \ar[r]^(.45){\tr} \ar[d]^{c} &
{ \TC(A;p) } \ar[d]^{c} \cr
{ K^B(A) } \ar[r]^(.47){\tr^B} &
{ \TC^B(A;p) } \cr
}$$
where the lower horizontal map is the map of Bass completed theories
induced by the cyclotomic trace, and where, by
Cor.~\ref{TCbasscomplete}, the right-hand vertical map is a weak
equivalence. This gives the desired extension of the cyclotomic trace
map to a map from the non-connective Bass completed $K$-theory
spectrum to the Bass completed topological cyclic homology
spectrum. In the following, we do not distinguish the topological
cyclic homology spectrum and its Bass completion and we write
$$\tr \colon K(A) \to \TC(A;p)$$
for the lower horizontal map in the diagram above. 

\section{The relative theorem}\label{relativesection}

The proofs of Thms.~\ref{relative} and~\ref{relativetheorem} are based
on the description in~\cite[Sect.~2]{h1} of the topological Hochschild
$\mathbb{T}$-spectrum of split square zero extension of rings. We
briefly recall this description.

We let $A$ be a unital associative ring, let $I \subset A$ be a
two-sided square zero ideal with quotient ring $B = A/I$ and assume
that the canonical projection of $A$ onto $B$ admits a ring
section. In this case, we recall from~\cite[Prop.~2.1]{h3} that there
is a wedge decomposition of $\mathbb{T}$-spectra 
$$\bigvee_{r \geqslant 1}T(B \ltimes I;r) \xto{\sim} T(A,I),$$
where the wedge sum ranges over the set of positive integers. We
remark that the spectrum $T(B \ltimes I;r)$ was denoted $T_r(B \oplus I)$ in
op.~cit. We write the induced direct sum decomposition of equivariant
homotopy groups as
$$\bigoplus_{r \geqslant 1} \TR_q^n(B \ltimes I;r;p) \xto{\sim}
\TR_q^n(A,I;p).$$
The Frobenius map
$\smash{ F \colon \TR_q^n(A,I;p) \to \TR_q^{n-1}(A,I;p) }$ preserves
the direct sum decomposition while the restriction map
$\smash{ R \colon \TR_q^n(A,I;p) \to \TR_q^{n-1}(A,I;p) }$ maps the
summand indexed by the positive integer $r$ divisible by $p$ to the
summand indexed by $r/p$ and annihilates the remaining
summands. Moreover, it follows from~\cite[Thm.~2.2]{hm} that there is
a long-exact sequence 
$$\cdots \to \mathbb{H}_q(C_{p^{n-1}},T(B \ltimes I;r)) \to
\TR_q^n(B \ltimes I;r;p) \xto{R}
\TR_q^{n-1}(B \ltimes I;r/p;p) \to \cdots$$
where the right-hand group is understood to be zero, if $p$ does not
divide $r$. The left-hand group is the $q$th group homology of
$C_{p^{n-1}}$ with coefficients in the $T(B \ltimes I;r)$ and is the abutment
of the strongly convergent spectral sequence
$$E_{s,t}^2 = H_s(C_{p^{n-1}},\TR_t^1(B \ltimes I;r;p)) \Rightarrow
\mathbb{H}_{s+t}(C_{p^{n-1}},T(B \ltimes I;r)).$$
It follows immediately from the definition of the
$\mathbb{T}$-spectrum $T(B \ltimes I;r)$ that the homotopy groups
$\TR_t^1(B \ltimes I;r;p) = \pi_tT(B \ltimes I;r)$ are zero, for $t < r-1$. Hence, the
spectral sequence and long-exact sequence above imply that the
restriction map
$$R \colon \TR_q^n(B \ltimes I;r;p) \to \TR_q^{n-1}(B \ltimes I;r/p;p)$$
is an isomorphism, for $q < r-1$, and an epimorphism, for $q =
r-1$. In particular, for every $n$ and $q$, only finitely many
summands in the direct sum decomposition of the group $\TR_q^n(A,I;p)$
are non-zero. 

\begin{lemma}\label{proabeliangrouplemma}Let $p$ be a prime number,
and let $X$ be a spectrum whose homotopy groups are $p$-primary
torsion groups of bounded exponent. Then, for all integers $q$, the
canonical map defines an isomorphism of pro-abelian groups
$$\pi_q(X) \xto{\sim} \{ \pi_q(X,\Z/p^v\Z) \}.$$
\end{lemma}

\begin{proof}The coefficient sequences give a short-exact sequence of
pro-abelian groups 
$$0 \to \{ \pi_q(X) \otimes \Z/p^v\Z \} \to \{ \pi_q(X,\Z/p^v\Z) \}
\to \{ \Tor(\pi_{q-1}(X),\Z/p^v\Z) \} \to 0$$
and the map of the statement factors through the canonical projection
$$\pi_q(X) \to \{ \pi_q(X) \otimes \Z/p^v\Z \}.$$
The latter map is an isomorphism of pro-abelian groups, since the
structure maps in the target pro-abelian group are isomorphisms, for
$v$ is strictly larger than the exponent of $\pi_q(X)$. Similarly, the
pro-abelian group on the right-hand side in the short-exact sequence
above is isomorphic to zero, since, for $v$ larger than the exponent
of $\pi_{q-1}(X)$, the structure map
$$\Tor(\pi_{q-1}(X),\Z/p^{2v}\Z) \to \Tor(\pi_{q-1}(X),\Z/p^v\Z)$$
is zero. This completes the proof.
\end{proof}

\begin{theorem}\label{TCrelative}Let $A$ be a unital associative ring
and $I \subset A$ a two-sided nilpotent ideal. Suppose that the prime
number $p$ is nilpotent in $A$. Then, for every integer $q$, the
canonical map defines an isomorphism of pro-abelian groups
$$\TC_q(A,I;p) \xto{\sim} \{ \TC_q^n(A,I;p) \}.$$
\end{theorem}

\begin{proof}We consider the diagram of canonical maps of pro-abelian groups
$$\xymatrix{
{ \TC_q(A,I;p) } \ar[r] \ar[d] &
{ \{ \TC_q^n(A,I;p) \} } \ar[d] \cr
{ \{ \TC_q(A,I;p,\Z/p^v\Z) \} } \ar[r] &
{ \{ \TC_q^n(A,I;p,\Z/p^v\Z) \}. } \cr
}$$
We have previously proved in~\cite[Thm.~2.1.1]{gh3} that the the lower
horizontal map is an isomorphism. Moreover, by Lemma~\ref{TCTR}, the
groups $\TC_q^n(A,I;p)$ are $p$-primary torsion groups of bounded
exponent, and hence Lemma~\ref{proabeliangrouplemma} shows that the
right-hand vertical map is an isomorphism. Therefore, the statement of
the theorem is equivalent to the statement the left-hand vertical map
is an isomorphism which, in turn, is equivalent to the statement that
the groups $\TC_q(A,I;p)$ are $p$-primary torsion groups of bounded
exponent.

We first assume that $I \subset A$ is a square zero ideal and that the
canonical projection of $A$ onto the quotient ring $B = A/I$ admits a
ring section. We use the wedge decomposition of $T(A,I)$ which we
recalled above to give a formula for the spectrum $\TC(A,I;p)$. First,
for the homotopy limit with respect to $F$, we find
$$\begin{aligned}
\TF(A,I;p) {} & = \holim_F \TR^n(A,I;p) \xleftarrow{\sim}
\holim_F \bigvee_{r \geqslant 1} \TR^n(B \ltimes I;r;p) \cr
{} & \xto{\sim} \prod_{r \geqslant 1} \holim_F \TR^n(B \ltimes I;r;p)
= \prod_{r \geqslant 1} \TF(B \ltimes I;r;p), \cr
\end{aligned}$$
where the second map is a weak equivalence since, for every $n$ and
$q$, the homotopy groups $\TR_q^n(B \ltimes I;r;p)$ are non-zero for only
finitely many $r$. We rewrite the product on the right-hand side as
$$\prod_{r \geqslant 1} \TF(B \ltimes I;r;p) = \prod_{j \in I_p} \prod_{v
  \geqslant 0} \TF(B \ltimes I;p^{v-1}j;p),$$
where $I_p$ denotes the set of positive integers not divisible by
$p$. Then the restriction map takes the factor indexed by $(j,v)$ to
the factor indexed by $(j,v-1)$. Hence, taking the homotopy equalizer
of the restriction map and the identity map, we obtain a weak
equivalence
$$\TC(A,I;p) \simeq \prod_{j \in I_p} \holim_R \TF(B \ltimes I;p^{v-1}j;p).$$
We argued in the discussion preceeding Thm.~\ref{TCrelative} that the
restriction map
$$R \colon \TR_q^n(B \ltimes I;p^{v-1}j;p) \to \TR_q^{n-1}(B \ltimes I;p^{v-2}j;p)$$
is an isomorphism, if $q+1 < p^{v-1}j$. Hence, the Milnor short-exact
sequence for the homotopy groups of a homotopy limit shows that, for
$1 \leqslant j \leqslant q+1$, the canonical projection induces an
isomorphism
$$\pi_q(\holim_R \TF(B \ltimes I;p^{v-1}j;p)) \xto{\sim} \TF_q(B \ltimes I;p^{s-1}j;p),$$
where $s = s_p(q,j)$ is the unique integer such that $p^{s-1}j
\leqslant q+1 < p^sj$, and that for $q+1 < j$, the homotopy group
on the left-hand side vanishes. Therefore, to show that $\TC_q(A,I;p)$
is a $p$-primary torsion group of bounded exponent, it suffices to
show that, for all integers $q$ and $s \geqslant 1$ and all $j \in
I_p$, $\TF_q(B \ltimes I;p^{s-1}j;p)$ is a $p$-primary torsion groups of
bounded exponent. Now, from~\cite[Thm.~2.2]{hm1}, we have the
following cofibration sequence of spectra
$$\mathbb{H}_{\boldsymbol{\cdot}}(C_{p^{n-1}},T(B \ltimes I;p^{s-1}j)) \to
\TR^n(B \ltimes I;p^{s-1}j;p) \xto{R}
\TR^{n-1}(B \ltimes I;p^{s-2}j;p),$$
and taking homotopy limits with respect to the Frobenius maps, we
obtain the following cofibration sequence of spectra
$$\holim_F\mathbb{H}_{\boldsymbol{\cdot}}(C_{p^{n-1}},T(B \ltimes I;p^{s-1}j)) \to
\TF(B \ltimes I;p^{s-1}j;p) \xto{R}
\TF(B \ltimes I;p^{s-2}j;p).$$
Hence, by induction on $s \geqslant 1$, it will suffice to prove that,
for all $s \geqslant 1$ and $j \in I_p$, the homotopy groups of the
homotopy limit on the left-hand term are $p$-primary torsion groups of
bounded exponent. We recall the spectral sequence
$$E_{s,t}^2 = H_s(C_{p^{n-1}},\TR_t^1(B \ltimes I;p^{s-1}j;p)) \Rightarrow
\mathbb{H}_{s+t}(C_{p^{n-1}},T(B \ltimes I;p^{s-1}j;p));$$
see~\cite[Sect.~4]{h7} for a detailed discussion. Since the spectral
sequence induces a finite filtration of the abutment, we obtain
a spectral sequence of pro-abelian groups
$$\{ E_{s,t}^2 \} = \{ H_s(C_{p^{n-1}},\TR_t^1(B \ltimes I;p^{s-1}j;p)) \} 
\Rightarrow \{ \mathbb{H}_{s+t}(C_{p^{n-1}},T(B \ltimes I;p^{s-1}j;p)) \},$$
where the pro-abelian groups are indexed by integers $n \geqslant 1$,
and where the structure maps in the pro-abelian groups are the
Frobenius maps. On $E^2$-terms, the Frobenius map induces the transfer
map in group homology
$$F \colon H_s(C_{p^{n-1}},\TR_t^1(B \ltimes I;p^{s-1}j;p)) \to
H_s(C_{p^{n-2}},\TR_t^1(B \ltimes I;p^{s-1}j;p))$$
which is readily evaluated~\cite[Lemma~6]{h7}. The result is that
there are isomorphisms of pro-abelian groups  
$$\{ E_{s,t}^2 \} \cong \begin{cases}
\TR_t^1(B \ltimes I;p^{s-1}j;p) & (\text{$s = 0$ or $s$ odd)} \cr
0 & (\text{otherwise}). \cr
\end{cases}$$
It follows that, for all integers $q$, $\{
\mathbb{H}_q(C_{p^{n-1}},T(B \ltimes I;p^{s-1}j)) \}$ is isomorphic to a
constant pro-abelian group and that this constant pro-abelian group is
a $p$-primary torsion group of bounded exponent. But then the
canonical map
$$\pi_q(\holim_F
\mathbb{H}_{\boldsymbol{\cdot}}(C_{p^{n-1}},T(B \ltimes I;p^{s-1}j))) \to
\{ \mathbb{H}_q(C_{p^{n-1}},T(B \ltimes I;p^{s-1}j)) \}$$
is an isomorphism of pro-abelian groups, and hence, the left-hand
group is a $p$-primary torsion group of bounded exponent as
desired. This completes the proof of the theorem in the case where $A$
is a split square zero extension.

We next let $I \subset A$ be any square zero ideal and show that, for
all integers $q$, the pro-abelian group $\{ \TC_q^n(A,I;p) \}$ is
isomorphic to a constant pro-abelian group. This implies that, for all
integers $q$, the map of the statement is an isomorphism of
pro-abelian groups. We choose a weak equivalence of simplicial rings
$$\epsilon_{A/I} \colon A/I[-] \to A/I$$
such that, for all $k \geqslant 0$, the ring $A/I[k]$ is a free unital
associative ring. Then the map $\epsilon_A \colon A[-] \to A$ defined
by the pull-back diagram of simplicial rings
$$\xymatrix{
{ A[-] } \ar[r] \ar[d]^{\epsilon_A} &
{ A/I[-] } \ar[d]^{\epsilon_{A/I}} \cr
{ A } \ar[r] &
{ A/I }
}$$
is again a weak equivalence. In this case, there is a spectral sequence
$$E_{s,t}^1 = \TC_t^n(A[s],I;p) \Rightarrow
\TC_{s+t}^n(A,I;p),$$
for every integer $n \geqslant 1$. Since the filtration of the group
in the abutment is finite, we obtain a spectral sequence of
pro-abelian groups
$$\{ E_{s,t}^1 \} = \{ \TC_t^n(A[s],I;p) \} \Rightarrow
\{ \TC_{s+t}^n(A,I;p) \}.$$
Moreover, since $A/I[k]$ is a free unital associative ring, the
canonical projection from $A[k]$ onto $A/I[k]$ admits a ring
section. Therefore, by the case considered earlier, we conclude that,
for all integers $s$ and $t$, the pro-abelian group $\{ E_{s,t}^1 \}$
is isomorphic to a constant pro-abelian group. This implies that, for
all integers $q$, the pro-abelian group $\{ \TC_q^n(A,I;p) \}$ is
isomorphic to a constant pro-abelian group. Indeed, the category of
constant pro-abelian groups is a full abelian subcategory of the
abelian category of all pro-abelian groups. This completes the proof
in the case where $I \subset A$ is a square zero ideal.

Finally, we show by induction on $m \geqslant 2$ that the map of the
statement is an isomorphism of pro-abelian groups, if $I \subset A$ is
a two-sided ideal with $I^m = 0$. The case $m = 2$ was proved
above. So let $m > 2$ and assume that the statement has been proved
for smaller $m$. We consider the long-exact sequence
$$\cdots \to
\TC_q(A,I^{m-1};p) \to
\TC_q(A,I;p) \to
\TC_q(A/I^{m-1},I/I^{m-1};p) \to \cdots.$$
By induction, the right and left-hand groups both are $p$-torsion
groups of bounded exponents, and therefore, so is the middle
group. This completes the proof.
\end{proof}

\begin{proof}[Proof of Thm.~\ref{relative}]The arithmetic
square~\cite[Prop.~2.9]{bousfield} gives rise to the following
long-exact sequence where the two products range over all prime
numbers:
$$\xy
(-47,0)*{\cdots};
(-29,0)*{K_q(A,I)};
(0,4.8)*{\prod_{\ell} K_q(A,I;\Z_{\ell})};
(0,-5)*{K_q(A,I) \otimes \Q};
(0,0)*{\oplus};
(38,0)*{\prod_{\ell} K_q(A,I;\Z_{\ell}) \otimes \Q};
(66,0)*{\cdots};
{ \ar (-37,0)*{};(-44,0)*{};};
{ \ar (-14,0)*{};(-21,0)*{};};
{ \ar (21,0)*{};(14,0)*{};};
{ \ar (62,0)*{};(55,0)*{};};
\endxy$$
The theorem of Goodwillie~\cite[Main Thm.]{goodwillie} identifies
$K_q(A,I) \otimes \Q$ with the relative negative cyclic homology group
$\HC_q^{-}(A \otimes \Q,I \otimes \Q)$ which, in turn, is zero, since
$p$ is nilpotent in $A$. Similarly, McCarthy's theorem~\cite[Main
Thm.]{mccarthy1} shows that, for every prime number $\ell$, the
cyclotomic trace map induces an isomorphism
$$K_q(A,I;\Z_{\ell}) \xto{\sim} \TC_q(A,I;\ell,\Z_{\ell}).$$
Since $p$ is nilpotent in $A$, this group vanishes, for $\ell \neq
p$. Moreover, Thm.~\ref{TCrelative} implies that the completion map
$\TC_q(A,I;p) \to \TC_q(A,I;p,\Zp)$ is an isomorphism and that the
common group is a $p$-primary torsion group of bounded exponent. We
conclude that the cyclotomic trace is an isomorphism
$$K_q(A,I) \xto{\sim} \TC_q(A,I;p)$$
and that $K_q(A,I)$ is a $p$-primary torsion group of bounded exponent
as stated.
\end{proof}

\begin{proof}[Proof of Thm.~\ref{relativetheorem}]The map of the
statement is equal to the composition
$$K_q(A,I) \to \TC_q(A,I;p) \to \{ \TC_q^n(A,I;p) \}$$
of the cyclotomic trace map and the canonical map. We saw in the
proof of Thm.~\ref{relative} above that the former map is an
isomorphism, and we proved in Thm.~\ref{TCrelative} that the latter
map is an isomorphism. The theorem follows.
\end{proof}

\section{The bi-relative theorem}\label{birelativesection}

In this section, we prove Thms.~\ref{birelative}
and~\ref{birelativetheorem} of the introduction. We view the ideal
$I$ as an associative ring without unit and form the associated
associative ring with unit $\Z \ltimes I$ defined to be the product
abelian group $\Z \times I$ with multiplication
$$(a,x) \cdot (a',x') = (aa',ax'+xa'+xx').$$
We recall that Suslin has proved in~\cite[Thm.~A]{suslin4} that, if
$\Tor_q^{\Z \ltimes I}(\Z,\Z)$ vanishes, for all $q > 0$, then the
bi-relative group $K_q(A,B,I)$ vanishes, for all $q$. We prove the
following result in a similar manner.

\begin{theorem}\label{suslintheorem}Let $f \colon A \to B$ be a map of
unital associative rings, let $I \subset A$ be a two-sided ideal and
assume that $f \colon I \to f(I)$ is an isomorphism onto a two-sided
ideal of $B$. Suppose that the pro-abelian group $\{ \Tor_q^{\Z
  \ltimes (I^m)}(\Z,\Z) \}$ is zero, for all positive integers
$q$. Then the canonical map
$$K_q(A,B,I) \to \{ K_q(A/I^m,B/I^m,I/I^m) \}$$
is an isomorphism of pro-abelian groups, for all integers $q$.
\end{theorem}

\begin{proof}We have a long-exact sequence of pro-abelian groups
$$\cdots \to \{ K_q(A,B,I^m) \} \to K_q(A,B,I) \to \{
K_q(A/I^m,B/I^m,I/I^m) \} \to \cdots$$
Since the pro-abelian group $\{ \Tor_q^{\Z \ltimes (I^m) }(\Z,\Z) \}$
is zero, for all $q > 0$, we conclude from the proof
of~\cite[Thm.~1.1]{gh4} that the left-hand pro-abelian group is zero,
for all integers $q$. The theorem follows.
\end{proof}

\begin{remark}\label{question}We do not know whether the canonical map
$$K_q(A,B,I) \to \{ K_q(A/I^m,B/I^m,I/I^m) \}$$
is an isomorphism of pro-abelian groups, if the assumption that
the pro-abelian groups 
$\{ \Tor_q^{\Z \ltimes ( I^m)}(\Z,\Z) \}$ be zero is omitted.
\end{remark}

We to show in Prop.~\ref{proexcision} below that hypotheses of
Thm.~\ref{suslintheorem} are satisfied, if the ideal $I$ can be
embedded as an ideal of a free unital $\Fp$-algebra. We first discuss
a slight generalization of the non-standard homological algebra in
Suslin's paper~\cite{suslin4}. Let $k$ be a commutative ring, and let
$\{I_m\}$ be a pro-non-unital associative $k$-algebra. We define a
\emph{left $\{I_m\}$-module} to be a pro-$k$-module $\{ M_m \}$ and,
for every $m \geqslant 1$, a left $I_m$-module structure on $M_m$ such 
that, for all $m \geqslant n \geqslant 1$, the diagram 
$$\xymatrix{
{ I_m \otimes_k M_m } \ar[r] \ar[d] &
{ M_m } \ar[d] \cr
{ I_n \otimes_k M_n } \ar[r] &
{ M_n, } \cr
}$$
where the horizontal maps are the module structure maps and the
vertical maps are induced from the structure maps in the
pro-$k$-modules $\{I_m\}$ and $\{M_m\}$, commutes. A
\emph{homomorphism} from the left $\{I_m\}$-module $\{M_m\}$ to the
left $\{I_m\}$-module $\{M_m'\}$ is defined to be a strict map of pro-$k$-modules $f
\colon \{ M_m \} \to \{M_m'\}$ such that, for all $m \geqslant 1$, the
map $f_m \colon M_m \to M_m'$ is $I_m$-linear. An \emph{extended left
$\{ I_m \}$-module} is a left $\{ I_m \}$-module of the form
$\{ I_m \otimes_k L_m \}$, where $\{ L_m \}$ is a pro-$k$-module. The
left $\{I_m\}$-module $\{P_m\}$ is defined to be \emph{pseudo-free} if
there exists an isomorphism of left $\{ I_m \}$-modules $\varphi
\colon \{ I_m \otimes_k L_m \} \to \{ P_m \}$ from an extended left $\{
I_m \}$-module such that, for all $m \geqslant 1$, $L_m$ is a free
$k$-module. If $\{ P_m \}$ is a pseudo-free left $\{ I_m \}$-module,
we say that the homomorphism $f \colon \{ P_m \} \to \{ M_m \}$ to the
left $\{ I_m \}$-module $\{ M_m \}$ is \emph{special}, if there exists
a strict map of pro-$k$-modules $g \colon \{ L_m \} \to \{ M_m \}$
such that, for all $m \geqslant 1$, the diagram
$$\xymatrix{
{ I_m \otimes_k L_m } \ar[r]_(.6){\sim}^(.6){\varphi_m}
\ar[d]^{\id \otimes g_m} & 
{ P_m } \ar[d]^{f_m} \cr
{ I_m \otimes_k M_m } \ar[r]^(.58){\mu_m} &
{ M_m } \cr
}$$
commutes.

\begin{lemma}\label{tormapvanishing}Let $k$ be a commutative ring and
$\{ I_m \}$ a pro-non-unital associative $k$-algebra. Then, for all $q
\geqslant 0$, the special homomorphism $f \colon \{ P_m \} \to \{ M_m
\}$ from the pseudo-free left $\{ I_m \}$-module $\{ P_m \}$ to the
left $\{ I_m \}$-module $\{ M_m \}$ induces the zero-map of
pro-$k$-modules
$$\{ \Tor_q^{k \ltimes I_m}(k, P_m) \} \to
\{ \Tor_q^{k \ltimes I_m}(k, M_m) \}.$$
\end{lemma}

\begin{proof}We have a commutative diagram
$$\xymatrix{
{ I_m \otimes_k L_m } \ar[r]_(.6){\sim}^(.6){\varphi_m}
\ar[d]^{\iota_m \otimes \id} & 
{ P_m } \ar[d]^{f_m} \cr
{ (k \ltimes I_m) \otimes_k L_m } \ar[r]^(.65){\tilde{f}_m} &
{ M_m } \cr
}$$
where $\tilde{f}_m((a,x) \otimes y) = ag_m(y) + xg_m(y)$. For $q = 0$,
the map $\iota_m$ induces the zero-map $k \otimes_{k \ltimes I_m} I_m \to 
k \otimes_{k \ltimes I_m} (k \ltimes I_m)$, and for $q > 0$, 
$$\Tor_q^{k \ltimes I_m}(k,(k \ltimes I_m) \otimes L_m) = 
\Tor_q^{k \ltimes I_m}(k,k \ltimes I_m) \otimes L_m = 0.$$
The lemma follows.
\end{proof}

\begin{prop}\label{iszero}Let $k$ be a commutative ring, let
$\{ I_m \}$ be a pro-non-unital associative $k$-algebra and assume
that, for all $q > 0$, the pro-$k$-module
$$\{ \Tor_q^{k \ltimes I_m}(k,k) \}$$
is zero. Let $F_q$, $q \geqslant 0$, be functors from the category of
left $\{ I_m \}$-modules to the category of pro-abelian groups and
assume that the following holds:

{\rm (i)} If $\{ M_m[-] \} \to \{ M_m \}$ is an augmented
simplicial left $\{ I_m \}$-module such that the associated chain
complex of pro-abelian groups is exact, then there is a spectral
sequence of pro-abelian groups
$$E_{s,t}^1 = F_t(\{M_m[s]\}) \Rightarrow F_{s+t}(\{M_m\})$$
such that the edge-homomorphism $F_t(\{M_m[0]\}) \to F_t(\{M_m\})$ is
equal to the map induced by the augmentation $\epsilon \colon \{
M_m[0] \} \to \{M_m\}$. 

{\rm (ii)} If $f \colon \{ P_m \} \to \{ M_m \}$ is a special
homomorphism then, for all $q \geq 0$, the induced map of pro-abelian
groups $f_* \colon F_q(\{ P_m \}) \to F_q(\{ M_m \})$ is zero.

\noindent Then, for every pseudo-free left $\{I_m\}$-module $\{P_m\}$
and every $q \geqslant 0$, the pro-abelian group $F_q(\{P_m\})$ is
zero.
\end{prop}

\begin{proof}This is proved as in~\cite[Prop.~1.7]{gh4} by induction
on $q \geqslant 0$. The additional assumption in loc.~cit.~that, for
every left  $\{ I_m \}$-module $\{ M_m \}$, the pro-abelian group
$F_0(\{ M_m \})$ be zero is unnecessary. Indeed, the proof of the
induction step given in loc.~cit.~also proves the case $q = 0$. 
\end{proof}

We remark that, in Prop.~\ref{iszero}, the functors $F_q$ are not
assumed to be additive. 

Let $\alpha \colon k' \to k$ a ring homomorphism, and let
$\{ I_m \}$ be a pro-non-unital associative $k$-algebra. We define $\{
I_m^{\alpha} \}$ to be $\{ I_m \}$ considered as a pro-non-unital
associative $k'$-algebra via $\alpha$ and call it the
\emph{associated} pro-non-unital associative $k'$-algebra. Similarly,
if $\{ M_m \}$ is a left $\{ I_m \}$-module, we define $\{
M_m^{\alpha} \}$ to be $\{ M_m \}$ 
considered as a left $\{ I_m^{\alpha} \}$-module via $\alpha$ and call it the
\emph{associated} left $\{ I_m^{\alpha} \}$-module. Then the left $\{ I_m
\}$-module $\{ P_m \}$ is pseudo-free if and only if the associated
left $\{ I_m^{\alpha} \}$-module $\{ P_m^{\alpha} \}$ is
pseudo-free, and the homomorphism $f \colon \{ P_m \} \to \{ M_m
\}$ from the pseudo-free left $\{ I_m \}$-module $\{ P_m \}$ to the
left $\{ I_m \}$-module $\{ M_m \}$ is special if and only if the
associated  homomorphism $f^{\alpha} = f \colon \{ P_m^{\alpha} \} \to
\{ M_m^{\alpha} \}$ of the associated left $\{ I_m^{\alpha}
\}$-modules is special.

\begin{prop}\label{changeofrings}Let $\alpha \colon k' \to k$ be a map of
commutative rings, let $\{ I_m \}$ a pro-non-unital associative
$k$-algebra, and let $\{ I_m^{\alpha} \}$ be the associated
pro-non-unital associative $k'$-algebra. Assume that the
pro-$k$-module
$$\{ \Tor_q^{k \ltimes I_m}(k,k) \}$$
is zero, for all integers $q > 0$. Then, the pro-$k'$-module
$$\{ \Tor_q^{k' \ltimes I_m^{\alpha}}(k',k') \}$$
is zero, for all integers $q > 0$.
\end{prop}

\begin{proof}The short-exact sequence of left 
$k' \ltimes I_m^{\alpha}$-modules 
$$0 \to I_m^{\alpha} \to k' \ltimes I_m^{\alpha} \to k' \to 0$$
induces a long-exact sequence of pro-abelian groups
$$\cdots \to
\{ \Tor_q^{k' \ltimes I_m^{\alpha}}(k',I_m^{\alpha}) \} \to
\{ \Tor_q^{k' \ltimes I_m^{\alpha}}(k',k' \ltimes I_m^{\alpha}) \} \to
\{ \Tor_q^{k' \ltimes I_m^{\alpha}}(k',k') \} \to
\cdots$$
It shows that the boundary map is an isomorphism of pro-abelian groups
$$\{ \Tor_q^{k' \ltimes I_m^{\alpha}}(k',k') \} \xto{\sim}
\{ \Tor_{q-1}^{k' \ltimes I_m^{\alpha}}(k',I_m^{\alpha}) \},$$
for all $q > 0$. Therefore, the statement will follow from
Prop.~\ref{iszero}, once we show that the functors $F_q$, $q \geqslant
0$, from the category of left $\{I_m\}$-modules to the category of
pro-abelian groups defined by
$$F_q(\{M_m\}) =
\{ \Tor_q^{k' \ltimes I_m^{\alpha}}(k', M_m^{\alpha}) \}$$
satisfy the hypotheses~(i)--(ii) of loc.~cit. Indeed, considered as a
left $\{ I_m^{\alpha} \}$-module, $\{I_m^{\alpha}\}$ is a pseudo-free.

To prove hypothesis~(i), we let $\epsilon \colon \{M_m[-]\} \to
\{M_m\}$ be an augmented simplicial left $\{I_m\}$-module and consider
the bi-simplicial symmetric spectrum
$$X[s,t] = B(Hk',H(k' \ltimes I_m^{\alpha}), HM_m^{\alpha}[s])[t] 
{}  = Hk' \wedge H(k' \ltimes I_m^{\alpha})^{\wedge t} \wedge
HM_m^{\alpha}[s] $$
given by the two-sided bar-construction of the Eilenberg-Mac~Lane
spectra associated with the ring $k' \ltimes I_m^{\alpha}$ and the
left and right modules $k'$ and $M_m^{\alpha}[s]$. We recall that the
three possible ways of forming the geometric realization of $X[-,-]$
lead to the same result in the sense that there are canonical
isomorphisms
$$\vert [s] \mapsto \vert [t] \mapsto X[s,t] \vert \vert \xto{\sim}
\vert [n] \mapsto X[n,n] \vert \xleftarrow{\sim}
\vert [t] \mapsto \vert [s] \mapsto X[s,t] \vert \vert.$$
For the left-hand side, the skeleton filtration gives rise to a
spectral sequence that converges to the homotopy groups of the
geometric  realization. By~\cite[Thm.~2.1]{elmendorfkrizmandellmay},
we obtain the identification
$$E_{s,t}^1 = \pi_t(\vert [i] \mapsto X[s,i] \vert) 
= \Tor_t^{k' \ltimes I_m^{\alpha}}(k',M_m^{\alpha}[s]).$$
For the right-hand side, the augmentation $\epsilon$ induces a weak
equivalence
$$\vert [s] \mapsto X[s,t] \vert \xto{\sim}
B(Hk',H(k'\ltimes I_m^{\alpha}),HM_m^{\alpha})[t]$$
which, by~\cite[Prop.~VII.3.6]{goerssjardine}, induces a weak
equivalence
$$\vert [t] \mapsto \vert [s] \mapsto X[s,t] \vert \vert \xto{\sim} 
\vert [t] \mapsto B(Hk',H(k'\ltimes I_m^{\alpha}),HM_m^{\alpha})[t] \vert.$$
We conclude from~\cite[Thm.~2.1]{elmendorfkrizmandellmay} that
$$\pi_q(\vert [t] \mapsto \vert [s] \mapsto X[s,t] \vert \vert)
= \Tor_q^{k' \ltimes I_m^{\alpha}}(k',M_m^{\alpha}).$$
Hence, we have a spectral sequence
$$E_{s,t}^1 = \Tor_t^{k' \ltimes I_m^{\alpha}}(k',M_m^{\alpha}[s])
\Rightarrow \Tor_{s+t}^{k' \ltimes I_m^{\alpha}}(k',M_m^{\alpha})$$
which proves hypothesis~(i). To prove hypothesis~(ii), let
$f \colon \{P_m\} \to \{M_m\}$ be a special homomorphism of left
$\{ I_m \}$-modules. Then $f^{\alpha} \colon \{ P_m^{\alpha} \} \to
\{ M_m^{\alpha} \}$ is a special homomorphism of left
$\{I_m^{\alpha}\}$-modules, and hence, Lemma~\ref{tormapvanishing}
shows  that, for $q \geqslant 0$, the induced map
$f_* \colon F_q(\{ P_m \}) \to F_q(\{ M_m \})$ is zero. This completes
the proof.
\end{proof}

\begin{prop}\label{proexcision}Let $I$ be a two-sided ideal in a free
unital associative $\Fp$-algebra. Then, for all positive integers $q$,
the following pro-abelian group is zero:
$$\{ \Tor_q^{\Z \ltimes (I^m)}(\Z,\Z) \}.$$
\end{prop}

\begin{proof}By Prop.~\ref{changeofrings}, it suffices to show that,
for all $q > 0$, the pro-abelian group
$$\{ \Tor_q^{\Fp \ltimes (I^m)}(\Fp,\Fp) \}$$
is zero. Since $\Fp$ is a field, the $\Tor$-groups in question are
canonically isomorphic to the homology groups of the normalized
bar-complex
$$B(\Fp, \Fp \ltimes I^m, \Fp) = 
\Fp \otimes_{\Fp \ltimes I^m} B(\Fp \ltimes I^m,\Fp)$$
which takes the form
$$\cdots \xto{d_m} (I^m)^{\otimes n} \xto{d_m} \cdots \xto{d_m}
(I^m)^{\otimes 2} \xto{d_m} I^m \xto{d_m} \Fp$$
with differential
$$d_m(x_1 \otimes \cdots \otimes x_n) = \sum_{i=1}^{n-1} 
(-1)^i x_1 \otimes \cdots \otimes x_ix_{i+1} \otimes \cdots \otimes
x_n.$$ 
We define a chain-homotopy from the chain map
$$\iota \colon B(\Fp, \Fp \ltimes I^{2m}, \Fp) \to
B(\Fp, \Fp \ltimes I^m, \Fp)$$
induced by the canonical inclusion to the chain map that is zero in
positive degrees and the identity map in degree $0$. Replacing $I$ by
$I^m$, we may assume that $m = 1$. Suppose that $I$ is a two-sided
ideal in the free unital associative $\Fp$-algebra $F$. Then, as a
left $F$-module, $I$ is free, and hence,
projective~\cite[Cor.~4.3]{cohn}. It follows that the multiplication  
$\mu \colon F \otimes I \to I$ has an $F$-linear section $\alpha
\colon I \to F \otimes I$. We remark that $\alpha$ restricts to an
$I$-linear map $\alpha \colon I^2 \to I \otimes I$. Then
$$s(x_1 \otimes \cdots \otimes x_n) = (-1)^n x_1 \otimes \cdots
\otimes x_{n-1} \otimes \alpha(x_n)$$
is the desired chain-homotopy. We conclude that, for $q > 0$, the
canonical inclusion of $I^{2m}$ into $I^m$ induces the zero map
$$\Tor_q^{\Fp \ltimes I^{2m}}(\Fp,\Fp) \to 
\Tor_q^{\Fp \ltimes I^m}(\Fp,\Fp).$$
This completes the proof.
\end{proof}

It was proved by Goodwillie~\cite[Lemma~I.2.2]{goodwillie} that the
relative $K$-theory of a nilpotent extension of simplicial rings may
be evaluated degreewise. We need the following analogous result for
bi-relative $K$-theory.

\begin{lemma}\label{resolution}Let $f \colon A \to B$ be a map of
unital associative rings and $I \subset A$ a two-sided ideal such that
$f \colon I \to f(I)$ is an isomorphism onto a two-sided ideal of
$B$. Let $f[-] \colon A[-] \to B[-]$ be a map of simplicial unital
associative rings and $I[-]$ a simplicial two-sided ideal of $A[-]$
such that $f[-] \colon I[-] \to f(I[-])$ is an isomorphism onto a
simplicial two-sided ideal of $B[-]$. In addition, let  $\epsilon_A 
\colon A[-] \to A$ and $\epsilon_B \colon B[-] \to B$ be weak
equivalences of simplicial unital associative rings such that 
$f \circ \epsilon_A = \epsilon_B \circ f[-]$ and such that $\epsilon_A$
restricts to a weak equivalence of simplicial associative rings
$\epsilon_I \colon I[-] \to I$. Then there is a natural spectral
sequence
$$E_{s,t}^1 = K_t(A[s],B[s],I[s]) \Rightarrow K_{s+t}(A,B,I).$$
\end{lemma}

\begin{proof}We have a spectral sequence
$$E_{s,t}^1 = K_t(A[s],B[s],I[s]) \Rightarrow \pi_{s+t}(\vert [n]
\mapsto K(A[n],B[n],I[n]) \vert)$$
obtained from the skeleton filtration of the geometric realization.
Hence, it suffices to construct a weak equivalence of the
symmetric spectrum $K(A,B,I)$ and the
symmetric spectrum $\vert [n] \mapsto K(A[n],B[n],I[n]) \vert$. The
definition of $K$-theory was extended to simplicial rings by
Waldhausen~\cite{waldhausen1}. Moreover, by op.~cit., Prop.~1.1, the 
extended functor preserves weak equivalence such that, in the case at
hand, the augmentations induce a weak equivalence 
$$\epsilon_* \colon K(A[-],B[-],I[-]) \xto{\sim} K(A,B,I).$$
We proceed to relate the left-hand side to the symmetric spectrum
given by geometric realization above. Let $\Delta[n][-]$ be the
simplicial standard $n$-simplex whose set of $m$-simplices
$\Delta[n][m]$ is the set non-decreasing maps from $[m]$ to
$[n]$. Then, if $S[-]$ is a simplicial ring, we define $S[-,-]$ to be
the bi-simplicial ring
$$S[m,n] = \Hom(\Delta[m][-] \times \Delta[n][-],S[-])$$
where the right-hand side is the set of maps of simplicial sets. 
For every $n$, the canonical projection gives rise to a map of
simplicial rings $\pr_1^* \colon S[-] \to S[-,n]$ which is a weak
equivalence. It induces a map of symmetric spectra
$$\pr_1^* \colon K(A[-],B[-],I[-]) \to
\vert [n] \mapsto K(A[-,n],B[-,n],I[-,n]) \vert$$
which is a weak equivalence by~\cite[Prop.~1.1]{waldhausen1}
and~\cite[Prop.~VII.3.6]{goerssjardine}. Similarly, we have the map of
simplicial rings $\pr_2^* \colon S[n] \to S[-,n]$, where the ring
$S[n]$ is considered as a constant simplicial ring, which induces a
map of symmetric spectra 
$$\pr_2^* \colon \vert [n] \mapsto K(A[n],B[n],I[n]) \vert \to
\vert [n] \mapsto K(A[-,n],B[-,n],I[-,n]) \vert.$$
The composition of the map of homotopy groups induced by $\pr_2^*$,
the inverse of the map of homotopy groups induced by the weak
equivalence $\pr_1^*$, and the map of homotopy groups induced by the
weak equivalence $\epsilon_*$ defines a natural map
$$f_K \colon \pi_q(\vert [n] \mapsto K(A[n],B[n],I[n]) \vert) \to
K_q(A,B,I)$$
and the lemma will follow, if we prove that this map is an
isomorphism. 

By the arithmetic square~\cite[Prop.~2.9]{bousfield}, it will suffice
to show the map $f$ induces an isomorphism of rational homotopy
groups and, for every prime number $p$, an isomorphism of homotopy
groups with $\Z/p\Z$-coefficients. For the rational homotopy groups,
we compare the bi-relative $K$-groups to the corresponding bi-relative
negative cyclic and cyclic homology
groups. By~\cite[Sect.~I.3]{goodwillie}, the latter are also defined
for simplicial rings. Hence, we may follow the proceedure above and
define maps $f_{\HC^{-}}$ and $f_{\HC}$ that makes the following
diagram commute:
$$\xymatrix{
{ \pi_q(|[k] \mapsto K(A[k],B[k],I[k])|)\otimes \Q } 
\ar[r]^(.59){f_K} \ar[d] &
{ K_q(A,B,I) \otimes \Q } \ar[d] \cr
{ \pi_q(|[k] \mapsto \HC^{-}(A[k] \otimes \Q, B[k] \otimes \Q, I[k]
\otimes \Q)|) } \ar[r]^(.61){f_{\HC^{-}}} &
{ \HC_q^{-}(A \otimes \Q, B \otimes \Q, I \otimes \Q) } \cr
{ \pi_{q-1}(|[k] \mapsto \HC(A[k],B[k],I[k])|) \otimes \Q}
\ar[r]^(.59){f_{\HC}} \ar[u] & 
{ \HC_{q-1}(A,B,I) \otimes \Q, } \ar[u] \cr
}$$
By theorems of Corti\~{n}as~\cite[Thm.~0.1]{cortinas} and
Cuntz-Quillen~\cite{cuntzquillen}, the vertical maps are
isomorphisms. Moreover, it follows immediately from the definition of
cyclic homology that the lower horizontal map is an
isomorphism. Finally, for the homotopy groups with
$\Z/p\Z$-coefficients, we compare the bi-relative $K$-groups to the
corresponding bi-relative topological cyclic homology groups. The
latter are defined by simplicial rings, and hence, we obtain the
following commutative diagram:
$$\xymatrix{
{ \pi_q(|[k] \mapsto K(A[k],B[k],I[k])|;\Z/p\Z) } 
\ar[r]^(.6){f_K} \ar[d] &
{ K_q(A,B,I;\Z/p\Z) } \ar[d] \cr
{ \{ \pi_q(|[k] \mapsto \TC^n(A[k],B[k],I[k];p)|;\Z/p\Z) \} } 
\ar[r]^(.6){f_{\TC}} &
{ \{ \TC_q^n(A,B,I;p,\Z/p\Z) \} } \cr
}$$
It follows from our previous result~\cite[Thm.~1]{gh4} that the
vertical maps are isomorphisms of pro-abelian groups. We must show
that the lower horizontal map is an isomorphism of pro-abelian
groups. We prove the stronger statement that the map of integral
homotopy groups
$$f_{\TC} \colon \pi_q(\vert [k] \mapsto \TC^n(A[k],B[k],I[k];p) \vert
\to \TC_q^n(A,B,I;p)$$
is an isomorphism, for all integers $q$ and $n \geqslant 1$. The
cofibration sequence
$$\TC^n(-;p) \to \TR^n(-;p) \to \TR^{n-1}(-;p) \to \Sigma \TC^n(-;p)$$
shows that it will suffice to show that the map
$$f_{\TR} \colon \pi_q(\vert [k] \mapsto \TR^n(A[k],B[k],I[k];p) \vert
\to \TR_q^n(A,B,I;p)$$
is an isomorphism, for all integers $q$ and $n \geqslant 1$. Moreover,
the cofibration sequence
$$\mathbb{H}_{\boldsymbol{\cdot}}(C_{p^{n-1}},T(-)) \to
\TR^n(-;p) \to \TR^{n-1}(-;p) \to
\Sigma \mathbb{H}_{\boldsymbol{\cdot}}(C_{p^{n-1}},T(-))$$
from~\cite[Thm.~2.2]{hm1} shows that it will suffice to show that
$f_{\TR}$ is an isomorphism, for $n = 1$ and for all $q$. Indeed,
the functor $\mathbb{H}_{\boldsymbol{\cdot}}(C_{p^{n-1}},-)$ commutes
with geometric realization and preserves weak equivalences. Finally,
the functor $\TR^1(-;p)$ is the topological Hochschild homology
functor $\THH(-)$, and we wish to show that
$$f_{\THH} \colon \pi_q(\vert [k] \mapsto \THH(A[k],B[k],I[k]) \vert
\to \THH_q(A,B,I)$$
is an isomorphism, for all integers $q$. This follows immediately from
the definition of topological Hochschild homology. 
\end{proof}

\begin{proof}[Proof of Thm.~\ref{birelative}]We first assume that
$(A,B,I)$ is a triple of $\Fp$-algebras and that $I$ admits an
embedding as a two-sided ideal in a free unital associative
$\Fp$-algebra. In this case, Prop.~\ref{proexcision} shows that the
pro-abelian group
$$\{ \Tor_q^{\Z \ltimes (I^m)}(\Z,\Z) \}$$
is zero, for every $q > 0$, and hence, Thm.~\ref{suslintheorem} shows
that the canonical map
$$K_q(A,B,I) \to \{ K_q(A/I^m, B/I^m, I/I^m) \}$$
is an isomorphism of pro-abelian groups. In particular, the group
$K_q(A,B,I)$ is a direct summand in $K_q(A/I^m,B/I^m,I/I^m)$, for $m$
large. Now, by the definition of the bi-relative groups, we have a
long-exact sequence
$$\cdots \to K_q(A/I^m,B/I^m,I/I^m) \to K_q(A/I^m,I/I^m) \to
K_q(B/I^m,I/I^m) \to \cdots$$
By Thm.~\ref{relative}, the middle and right-hand groups are
$p$-primary torsion groups of bounded exponent, and hence, the same
holds for the left-hand group. This shows that $K_q(A,B,I)$ is a
$p$-primary torsion group of bounded exponent as stated.

In the general case, we let $N$ be a positive integer such that $p^N$
annihilates $A$ and choose a diagram of simplicial unital associative
$\Z/p^N\Z$-algebras
$$\xymatrix{
{ B[-] } \ar[r] \ar[d]^{\epsilon_B} &
{ B/I[-] } \ar[d]^{\epsilon_{B/I}} &
{ A/I[-] } \ar[l]_{\bar{f}[-]} \ar[d]^{\epsilon_{A/I}} \cr
{ B } \ar[r] &
{ B/I } &
{ A/I, } \ar[l]_{\bar{f}} \cr
}$$
such that the vertical maps are weak equivalences and such that the
simplicial algebras in the top row are degree-wise free unital
associative $\Z/p^N\Z$-algebras. We then consider the diagram of
simplicial $\Z/p^N\Z$-algebras
$$\xymatrix{
{ I[-] } \ar[r] \ar@{=}[d] &
{ A[-] } \ar[r] \ar[d]^{f[-]} &
{ A/I[-] } \ar[d]^{\bar{f}[-]} \cr
{ I[-] } \ar[r] &
{ B[-] } \ar[r] &
{ B/I[-] } \cr
}$$
where the right-hand square is a pull-back square, and where the upper
and lower left-hand horizontal maps are the kernels of the upper and
lower right-hand horizontal maps. The maps $\epsilon_B$,
$\epsilon_{B/I}$, and $\epsilon_{A/I}$ above induce maps of simplicial
associative $\Z/p^N\Z$-algebras $\epsilon_A \colon A[-] \to A$ and
$\epsilon_I \colon I[-] \to I$ which are weak equivalences. It follows
from Lemma~\ref{resolution} that there is a spectral sequence
$$E_{s,t}^1 = K_t(A[s],B[s],I[s]) \Rightarrow K_{s+t}(A,B,I).$$
Hence, it suffices to show that $K_q(A[s],B[s],I[s])$ is a $p$-primary
torsion group of bounded exponent, for all integers $q$ and
$s \geqslant 0$. Since $B/I[s]$ is a free unital associative
$\Z/p^N\Z$-algebra, it follows that the canonical projections of
$A[s]$ onto $A/I[s]$ and $B[s]$ onto $B/I[s]$ have ring
sections. Therefore, the diagram
$$\xymatrix{
{ I[s]/pI[s] } \ar[r] \ar@{=}[d] &
{ A[s]/pA[s] } \ar[r] \ar[d] &
{ A/I[s]/pA/I[s] } \ar[d] \cr
{ I[s]/pI[s] } \ar[r] &
{ B[s]/pB[s] } \ar[r] &
{ B/I[s]/pB/I[s] } \cr
}$$
is a diagram of associative $\Fp$-algebras such that the right-hand
square is a pull-back square and such that the upper and lower
left-hand horizontal maps are the kernels of the upper and lower
right-hand horizontal maps. Now, since  $B[s]/pB[s]$ is a
free unital associative $\Fp$-algebra, it follows from the special
case considered first that $K_q(A[s]/pA[s],B[s]/pB[s],I[s]/pI[s])$ is
a $p$-primary torsion group of bounded exponent, for all integers $q$
and $s \geqslant 0$. Moreover, the mapping fiber of the map
$$K(A[s],B[s],I[s]) \to K(A[s]/pA[s],B[s]/pB[s],I[s]/pI[s])$$
induced by the canonical projections is canonically isomorphic to the
iterated mapping fiber of the square of relative $K$-theory spectra
$$\xymatrix{
{ K(A[s],pA[s]) } \ar[r] \ar[d] &
{ K(A/I[s],pA/I[s]) } \ar[d] \cr
{ K(B[s],pB[s]) } \ar[r] &
{ K(B/I[s],pB/I[s]). } \cr
}$$
It follows from Thm.~\ref{relative} above that the homotopy groups of
this iterated mapping fiber are $p$-primary torsion groups of bounded
exponent. Hence, $K_q(A[s],B[s],I[s])$ is a $p$-primary torsion group
of bounded exponent, for all integers $q$ and $s \geqslant 0$. We
conclude that, for every integer $q$, the group $K_q(A,B,I)$ is a
$p$-primary torsion group of bounded exponent as stated.
\end{proof}

\begin{proof}[Proof of Thm.~\ref{birelativetheorem}]We consider the
following diagram of pro-abelian groups:
$$\xymatrix{
{ K_q(A,B,I) } \ar[r] \ar[d] &
{ \{ K_q(A,B,I;\Z/p^v\Z) \} } \ar[d] \cr
{ \{ \TC_q^n(A,B,I;p) \} } \ar[r] &
{ \{ \TC_q^n(A,B,I;p,\Z/p^v\Z) \} } \cr
}$$
It follows from Thm.~\ref{birelative} and
Lemma~\ref{proabeliangrouplemma} that the 
upper horizontal map is an isomorphism. Similarly,
Lemma~\ref{proabeliangrouplemma} shows that the lower horizontal map
is an isomorphism. Finally, we proved in~\cite[Thm.~1]{gh4} that the
right-hand vertical map is an isomorphism. Hence, the left-hand
vertical map is an isomorphism as stated. 
\end{proof}

\providecommand{\bysame}{\leavevmode\hbox to3em{\hrulefill}\thinspace}
\providecommand{\MR}{\relax\ifhmode\unskip\space\fi MR }
\providecommand{\MRhref}[2]{%
  \href{http://www.ams.org/mathscinet-getitem?mr=#1}{#2}
}
\providecommand{\href}[2]{#2}

\end{document}